\documentclass[12pt]{amsart}

\usepackage{amssymb,amscd}
\usepackage{amsbsy,amsmath,amsfonts,amsthm,extarrows,bm}
\usepackage[english]{babel}
\usepackage{graphicx}
\usepackage[all]{xy}
\usepackage[latin1]{inputenc}

\newtheorem{lemma}{Lemma}

\newtheorem{Atheorem}{Theorem}[section]

\newenvironment{sproof}[1]
{\begin{proof}[#1]} {\end{proof}}

\newcommand{\Z}{\mathbb Z}
\newcommand{\R}{\mathbb R}
\newcommand{\Q}{\mathbb Q}

\newcommand{\GL}{\textnormal{GL}}
\newcommand{\SL}{\textnormal{SL}}
\newcommand{\D}{\textnormal{D}}

\newcommand{\rk}{\textnormal{rk}}

\newcommand{\GE}{\textnormal{GE}}
\newcommand{\U}{\textnormal{Um}}
\newcommand{\E}{\textnormal{E}}

\newcommand{\pr}{\operatorname{pr}}

\newcommand{\br}[1]{\lbrack #1 \rbrack}
\newcommand{\rb}{\mathbf{r}}
\newcommand{\mD}{\mathcal{D}}
\newcommand{\mO}{\mathcal{O}}
\newcommand{\mOD}{\mathcal{O}(\mathcal{D})}
\newcommand{\mODe}{\mathcal{O}(\mathcal{D} \setminus \{e\})}

\newcommand{\PZd}{\prod_{d \in \mD} \Z \lbrack \zeta_d \rbrack}

\newcommand{\ZX}{ \Z \lbrack  X \rbrack}
\newcommand{\Zd}{ \Z \lbrack \zeta_d \rbrack}
\newcommand{\Ze}{ \Z \lbrack \zeta_e \rbrack}
\newcommand{\ZC}{ \mathbb{Z} \lbrack C \rbrack}

\newcommand{\cC}{\vert C \vert}

\title{On quotients of generalized Euclidean group rings}
\address{EPFL ENT CBS BBP/HBP. Campus Biotech. B1 Building, Chemin des mines, 9\\Geneva 1202, Switzerland}
\email{luc.guyot@epfl.ch}
\author{Luc Guyot}
\date{\today}

\begin{document}
\maketitle
\begin{abstract}
Let $R = \ZC$ be the integral group ring of a finite cyclic group $C$. 
Dennis et al. proved that $R$ is a generalized Euclidean ring in the sense of P. M. Cohn, i.e., $\SL_n(R)$ is generated by the elementary matrices for all $n$.
We prove that every proper quotient of $R$ is also a generalized Euclidean ring. 
\end{abstract}

\section{Introduction} \label{SecIntro}
Let $R$ be an associative ring with identity. For $n \ge 2$, we define $\E_n(R)$ as the subgroup of $\GL_n(R)$ generated by the elementary matrices, i.e., the matrices which differ from the identity by a single off-diagonal element. The ring $R$ is termed a $\GE_n$-ring if 
$\GL_n(R) = \D_n(R)\E_n(R)$, where $\D_n(R)$ denotes the subgroup of invertible diagonal matrices. Following P. M. Cohn, we term $R$ a \emph{generalized Euclidean ring}, or concisely a $\GE$-ring , if $R$ is a $\GE_n$-ring for every $n \ge 2$, a property notoriously shared by Euclidean rings. 
If $R$ is commutative, Whitehead's lemma implies that $\D_n(R) \cap \SL_n(R) \subset \E_n(R)$. Therefore the property $\GE_n$ is equivalent to $\SL_n(R) = \E_n(R)$ in this case. Let $I$ be an ideal of a commutative ring $R$ and let $\varphi_n$ be the natural map 
$\SL_n(R) \rightarrow \SL_n(R/I)$. The natural map $\E_n(R) \rightarrow \E_n(R/I)$ is surjective and if $\varphi_n$ is surjective too, then $R/I$ is certainly a $\GE_n$-ring whenever $R$ is. 
But $\varphi_n$ can fail to be surjective for all $n \ge 2$ as is the case for $R = \R\br{X, Y}$ and $I$ the principal ideal generated by $X^2 + Y^2 -1$ (see   \cite[Example 13.5]{Mil71} and \cite[Proposition I.8.12 and Theorem VI.4.5]{Lam06}). Thus the inheritance of property $\GE_n$ by homomorphic images is not automatic.
In the previous example $R$ is a $\GE_3$-ring whereas $R/I$ is not a $\GE_n$-ring for any $n \ge 2$.
We can actually produce a non-commutative $\GE$-ring having a non-$\GE_2$ quotient. 
Indeed, a non-commutative free associative algebra over a field $k$
 is a $\GE$-ring and maps homomorphically onto a free commutative algebra over $k$ which is not a $\GE_2$-ring by \cite[Theorem 3.4 and Proposition 7.3]{Coh66}.

By contrast, the class of commutative $\GE$-rings is known to be stable under taking quotients by the ideals contained in the Jacobson radical \cite[Proposition 5]{Gel77}.
In addition, the class of quasi-Euclidean rings, i.e., the $\GE$-rings which are moreover Hermite rings in the sense of Kaplansky \cite{Kap49}, is stable under taking quotients by any ideals \cite{AJLL14}. 

The integral group ring $\ZC$ of a finite cyclic group $C$ was shown to be a $\GE$-ring in \cite{DMV84}. If $C$ is not trivial, then $\ZC$ has a non-principal ideal, and hence is not quasi-Euclidean. The purpose of this article is to establish
\begin{Atheorem} \label{ThZCGE}
Let $C$ be a finite cyclic group. Then every homomorphic image of $\Z \br{C}$ is a $\GE$-ring.
\end{Atheorem}

Our result has some bearing in combinatorial group theory. Let $R$ be quotient of $\ZC$ and let $M$ be a finitely generated free $R$-module with $C$ a finite cyclic group. Then the image of $C$ in $R$ is a subgroup of the unit group of $R$ so that $C$ acts naturally on $M$ by automorphisms. 
Theorem \ref{ThZCGE} allows us to show that the semi-direct product $G = M \rtimes C$ has only one Nielsen class of generating $k$-tuples for every $k \ge \rk(G) + 1$, where $\rk(G)$ denotes the minimal number of generators of $G$ \cite[Corollary 4]{Guy16b}.

Our proof of Theorem \ref{ThZCGE} reuses extensively the techniques developed in \cite{DMV84}. 
Lemma \ref{LemBottomUp} below compiles all criteria from the latter article which render induction possible.
The proof of these criteria is faithful to the original. We provide it for the reader's convenience. 
A notable difference in our case is that we need to extend the induction basis  
to the integer ring $\Zd$ of the cyclotomic field $\Q(\zeta_d)$ ($\zeta_d = e^{\frac{2i\pi}{d}}$) for all positive rational integers $d \notin \{1, 2, 3, 4, 6\}$. 
For this we can rely on \cite{Vas72}. The basis of our induction also comprises a set of $31$ small quotients of $\ZX$ parametrized by the non-empty subsets of $\{1, 2, 3, 4, 6\}$. These quotients can in turn be handled by induction, still using Lemma \ref{LemBottomUp}.\\

\paragraph{\textbf{Acknowledgments}.}
The author thanks Pierre de la Harpe and Tatiana Smirnova-Nagnibeda for encouragements 
and comments made on preliminary versions of this paper.

\section{Proof}
Our first step consists in the following observation: we only need to prove that $\GE_2$ holds for the homomorphic images of $\ZC$.
This is the assertion $(ii)$ of this first lemma which collects elementary facts for later use:
\begin{lemma} \label{LemGECriteria}
Let $R$ be a commutative ring with identity.
Then the following assertions hold:
\begin{itemize}
\item[$(i)$] If the stable rank of $R$ is $1$, then $R$ is a $\GE$-ring.
\item[$(ii)$] If the stable rank of $R$ is at most $2$ (e.g., the Krull dimension of $R$ is at most $1$), then $R$ is a $\GE$-ring if and only if it is a $\GE_2$-ring.
\item[$(iii)$] $R$ is a $\GE_2$-ring if and only if $\E_2(R)$ acts transitively on the set
$$\U_2(R) = \{ (r, s) \in R^2 \, \vert \,  rR + sR = R\}$$ of unimodular pairs of $R$.
\end{itemize}
\end{lemma}
$\square$

A semilocal ring has stable rank $1$ \cite[Corollary 6.5]{Bas64}. 
As a result semilocal rings, and Artinian rings in particular, are $\GE$-rings.
All rings considered in the proof of Theorem \ref{ThZCGE} are quotients of $\ZC$ and hence have stable rank at most $2$ \cite[Corollary 6.7.4]{McCR87}.

Our second step consists in a reduction to quotients of $\ZC$ of a special kind. To this end, we will need the following stability properties:
\begin{lemma} \label{LemStabilityOfGE}
Let $R$ be a commutative ring with identity.
Then the following assertions hold:
\begin{itemize}
\item[$(i)$]  Let $J$ be an ideal contained in the Jacobson radical of $R$. Then $R$ is a $\GE$-ring if and only if $R/J$ is a $\GE$-ring
\cite[Proposition 5]{Gel77}.
\item[$(ii)$] If $R/I$ is a $\GE$-ring for some finite ideal $I$ of $R$, then $R$ is a $\GE$-ring.
\end{itemize}
\end{lemma}
\begin{proof}
$(i)$ See aforementioned reference.
$(ii)$. Let $n \ge 2$ and let $A \in \SL_n(R)$. As $R/I$ is a $\GE$-ring, we can find $E \in \E_n(R)$ such that 
$B = AE^{-1}$ lies in $\SL_n(R, I) \Doteq \ker(\SL_n(R) \twoheadrightarrow \SL_n(R/I))$. The coefficients of $B$ lie in the subring $K$ of $R$ generated by $I$ and $1_R$, the identity of $R$. Therefore it suffices to show that $K$ is a $\GE$-ring. If the prime field of $R$ is finite, then $K$ is finite and hence Artinian. As $K$ is a $\GE$-ring in this case, we can assume now that the prime field of $R$ is $\Z$. We have then $\Z 1_R \cap I = 0$, which implies that the additive group of $K$ is isomorphic to $\Z \times I$. Let $e$ be the exponent of the additive group of $I$. Then $K/eK$ is finite and hence a $\GE$-ring. We can find $E' \in \E_n(K)$ such that $C = BE'^{-1}$ lies in 
$\SL_n(K, eK) \subset \SL_n(\Z 1_R)$. As $\Z 1_R \simeq \Z$ is Euclidean, we have $C \in \E_n(\Z 1_R)$, which completes the proof. 
\end{proof}

We now introduce the special $\GE_2$-quotients of $\ZX$ to which the study of $\ZC$ can be reduced.
Given a rational integer $d > 0$ we let $\lambda_d: \ZX \rightarrow \Zd$ be the ring homomorphism induced 
by the map $X \mapsto \zeta_d$.
Given a set $\mD$ of positive rational integers, we define $\lambda_{\mD} : \ZX \rightarrow \PZd$ by 
$\lambda_{\mD} =  \prod_{d \in \mD} \lambda_d$ and set $\mOD = \lambda_{\mD}(\ZX)$.

\begin{lemma} \label{LemOD}
Let $\mD$ be a finite and non-empty set of positive rational integers.
Then  $\mO(\mD)$ is a $\GE_2$-ring.
\end{lemma}

We postpone the proof of Lemma \ref{LemOD} and prove the theorem first.
\begin{sproof}{Proof of Theorem \ref{ThZCGE}}
Set $n = \cC$ and let $R$ be an homomorphic image of $\ZX$. There exists an ideal $I$ of $\ZX$ containing $X^n - 1$
 and such that $R = \ZX /I$.
If $R$ is finite, then $R$ is Artinian and hence a $\GE$-ring. Thus we can assume that $R$ is infinite. 
By Lemma \ref{LemStabilityOfGE}.$i$ we can also assume that $R$ is reduced. 
Since $R$ is Noetherian, the ideal $I$ is the intersection of finitely many prime ideals of $\ZX$. Those ideals contain $X^n - 1 = \prod_{d\, \vert \, n} \Phi_d(X)$, therefore they are either principal ideals generated by cyclotomic polynomials or maximal ideals. In the latter case the index in $R$ is necessarily finite. Consider then the finite and non-empty set $\mD$ of all positive rational integers such that $I \subset \Phi_d(X) \ZX$ and let $\pi: \ZX \twoheadrightarrow R$ be the natural ring epimorphism. 
Using the quotients of $\ZX$ by the prime ideals containing $I$, we easily see that there is a finite direct sum $K$ of finite fields, an injective ring homomorphism 
$\lambda: R \rightarrow \PZd \times K$ and a surjective ring homomorphism $\kappa: R \twoheadrightarrow K$ such that 
$\lambda \circ \pi = \lambda_{\mD} \times \kappa$.
Then the projection $\PZd \times K \twoheadrightarrow \PZd$ induces a ring epimorphism from
$R$ onto $\mO(\mD)$ whose kernel is $R \cap (\{0\} \times K)$, $R$ being identified with $\lambda(R)$. 
Applying Lemma \ref{LemStabilityOfGE}.$ii$ yields the conclusion.
\end{sproof}

The remainder of this section is devoted to the proof of Lemma \ref{LemOD}.

\begin{lemma} \label{LemBottomUp}
Let $\mD$ be a finite set of positive rational integers with at least two elements. Let $e \in \mD$ and
set $\eta_e = \prod_{d \in \mD \setminus \{e \}} \Phi_d(\zeta_e)$.
Assume that $\mODe$ is a $\GE_2$-ring and that at least one of the following holds:
\begin{itemize}
\item[$(i)$] 
$e \notin \{1, 2, 3, 4, 6\}$ and $\frac{d}{e}$ and is not a positive power of a prime for any
$d \in \mD \setminus \{e\}$.
\item[$(ii)$]
$e \in \{1, 2, 3, 4, 6\}$ and $\eta_e \in (\Ze)^{\times}$.
\item[$(iii)$]
$e \in \{1, 2\}$ and $\eta_e = \pm 3$.
\item[$(iv)$]
$e \in \{1, 2, 3, 6\}$ and $\eta_e \Ze = 2 \Ze$.
\item[$(v)$]
$e \in \{3, 4, 6\}$ and $\eta_e\Ze = (1 - \zeta_e)^k \Ze$ for some $k \in \{1, 2\}$.
\end{itemize}
 Then $\mOD$ is a $\GE_2$-ring.
\end{lemma}

\begin{proof}
Assertion $(i)$ is \cite[Lemma 3.5]{DMV84}. For the remaining assertions 
we identify $\mOD$ with its image in $\mODe \times \Ze$ and let 
$\rb = \left( \begin{pmatrix} a \\ c \end{pmatrix},  \begin{pmatrix} b \\ d \end{pmatrix}\right) \in \U_2(R)$.
Since the projection $\pr_e: \mOD \rightarrow \mODe$ is surjective and $\mODe$ is a $\GE_2$-ring, after multiplication 
by some $E \in \E_2(\mOD)$ we can assume that 
\begin{equation} \label{EqAB}
a = 1,\,b = 0
\end{equation} 
As $\ker(\pr_e)$ maps onto
$\eta_e \Ze$ via the projection $\mOD \twoheadrightarrow \Ze$, we have 
\begin{align} \label{EqCD}
c - 1, d \in \eta_e \Ze \\
\notag \mOD \cap (\{0\} \times \Ze) = \{0\} \times \eta_e\Ze.
\end{align}
Conditions (\ref{EqAB}) and (\ref{EqCD}) remain valid after multiplication by any product of matrices in
$$
\begin{pmatrix}
1 &&  \{0\} \times \eta_e\Ze \\
0 && 1
\end{pmatrix}
\bigcup 
\begin{pmatrix}
1 &&  0 \\
\mOD && 1
\end{pmatrix}.
$$
The projection of these of matrices to the $\Ze$ coordinate is
\begin{equation} \label{EqReductionMatrices}
\begin{pmatrix}
1 &&  \eta_e\Ze \\
0 && 1
\end{pmatrix}
\bigcup 
\begin{pmatrix}
1 &&  0 \\
\Ze && 1
\end{pmatrix}
\end{equation}

If $(ii)$ holds then matrices of the form (\ref{EqReductionMatrices}) generate $\E_2(\Ze)$. As $\Ze$ is a Euclidean ring for 
$e \in \{1, 2, 3, 4, 6\}$ such matrices can therefore be used to reduce $(c, d)$ to $(1, 0)$, reducing thus $\rb$ to $(1, 0)$.

If either $(iii)$ or $(iv)$ holds, then \cite[Lemma 4.1]{DMV84} applies, 
so that multiplication by matrices in (\ref{EqReductionMatrices})
can reduce either $c$ or $d$ to $0$. Since $c - 1 \in \eta_e \Ze$, it must be $d$. 
Assuming now that $d = 0$, the unimodularity of $(c, d)$ yields
$$
c \in (1 + \eta_e\Ze) \cap (\Ze)^{\times} \subset \{\pm 1\}.
$$
If $c = -1$, then $\eta_e \Ze = 2 \Ze$ must hold and the identity
$$
\begin{pmatrix}
c && d
\end{pmatrix}
\begin{pmatrix}
1 && 2 \\
0 && 1
\end{pmatrix}
\begin{pmatrix}
1 &&  0 \\
-1 && 1
\end{pmatrix}
\begin{pmatrix}
1 && 2 \\
0 && 1
\end{pmatrix}
=
\begin{pmatrix}
1 && 0
\end{pmatrix}
$$
shows that $(c, d)$ can be reduced to $(1, 0)$ using matrices in (\ref{EqReductionMatrices}).

If $(v)$ holds then $(\Ze, 1 - \zeta_e)$ is a Euclidean pair with respect to the complex modulus, so that \cite[Lemma 4.2]{DMV84} applies and leads to the desired conclusion.
\end{proof}

\begin{sproof}{Proof of Lemma \ref{LemOD}}
We apply inductively Lemma \ref{LemBottomUp}.$i$ by choosing at each step the largest member $e$ of $\mD$ which is not in 
$S = \{1, 2, 3, 4, 6\}$. In this way $\mD$ is reduced to either a subset of $S$ or a singleton $\{d\}$ with $d \notin S$. In the latter case, $\mOD = \mO(\{d\}) = \Zd$ is a $\GE_2$-ring by \cite{Vas72}.
We also observe that $\mO(\{d\})$ is a $\GE_2$-ring for all $d$ in $S$. Indeed, the rings $\Z$, $\Z \br{i}$ and $\Z\br{\zeta_3}$ are Euclidean rings with respect to the complex modulus. Applying assertions $(ii)$ to $(iv)$ of Lemma \ref{LemBottomUp}, we can prove the result inductively for all pairs, triples and quadruples of elements of $S$, except the quadruple $\{1, 2, 3, 6\}$ for which we apply \cite[Proposition 5.1]{DMV84}. 
Eventually, we apply Lemma \ref{LemBottomUp}.$v$ (or equivalently \cite[Proposition 5.10]{DMV84}) to prove the result for 
$\mD =S$.
\end{sproof}

\bibliographystyle{alpha}
\bibliography{Biblio}

\end{document}